\documentclass[11pt]{article}

\usepackage[utf8]{inputenc}
\usepackage{amsmath,amssymb,hyperref,array,xcolor,multicol,verbatim,mathpazo}
\usepackage[normalem]{ulem}
\usepackage[pdftex]{graphicx}
\usepackage{amssymb,amsfonts,amsmath,amsthm,amscd,mathrsfs,bbold}
\usepackage{bbm}
\usepackage[english]{babel}
\usepackage[utf8]{inputenc}
\usepackage{amsmath}
\usepackage{graphicx}
\usepackage[colorinlistoftodos]{todonotes}
\usepackage{amssymb}
\usepackage{wasysym}
\usepackage[none]{hyphenat}
\usepackage{thmtools}

\usepackage{amsmath}
\usepackage{amsthm}
\usepackage{amssymb}

\newtheorem{theorem}{Theorem}[section]
\newtheorem{corollary}[theorem]{Corollary}
\newtheorem{lemma}[theorem]{Lemma}
\newtheorem{proposition}[theorem]{Proposition}

\newtheorem{claim}[theorem]{Claim}
\newtheorem{question}[theorem]{Question}

\begin{document}


\title{Waiter-Client Maximum Degree Game }
\author{Michael Krivelevich 
\thanks{
Sackler School of Mathematical Sciences,
Tel Aviv University, Tel Aviv 6997801, Israel.
Email: {\tt krivelev@post.tau.ac.il}.  Research supported in part by USA-Israeli BSF grant 2014361 and by ISF grant 1261/17.}
\and 
Nadav Trumer
\thanks{
Sackler School of Mathematical Schiences,
Tel Aviv University, Tel Aviv 6997801, Israel.
Email: {\tt nadavtru@mail.tau.ac.il}.}
}
\maketitle

\begin{abstract}
For integers $n, D, q$ we define a two player perfect information game with no chance moves called the Waiter-Client Maximum Degree game. In this game, two players (Waiter and Client) play on the edges of $K_n$ as follows: in each round, Waiter offers $q+1$ edges which have not been previously offered. Client then claims one of these edges, and Waiter claims the rest. When less than $q+1$ edges which have not been offered remain, Waiter claims them all and the game ends. After the game ends, Client wins if in the graph of his edges, there is no vertex with degree at least $D$, and Waiter wins otherwise. For various values of $q = q(n)$, we study the maximum degree of Client's graph obtained by perfect play. We determine the asymptotic value of Client's maximum degree for the cases $q = o\left( \frac{n}{\ln n} \right)$ and $q = \omega\left( \frac{n}{\ln n} \right)$.
For the unbiased case $q=1$, we prove that when both players play perfectly the maximum degree $D$ Client achieves satisfies: $D = \frac{n}{2} + \Theta(\sqrt{n \ln n})$.
\end{abstract}

\section{Introduction} \label{introuction_section}

Positional games are, generally speaking, games of complete information with no random moves played by two players (and therefore, if a draw is not possible, one of the players has a deterministic winning strategy). The study of such games goes back to the seminal papers of Hales and Jewett \cite{Hales-Jewett}, of Lehman \cite{Lehman}, and of Erd\H{o}s and Selfridge \cite{Erdos-Selfridge}. It has developed over the years and has become a recognized area of study in combinatorics (see for example, the monograph of Beck \cite{Beck Monograph}, and the monograph of Hefetz, Krivelevich, Stojakovi\'{c} and Szab\'{o} \cite{Positional Games Monograph}). 

The most well-studied positional games are the so-called \textbf{Maker-Breaker} games, in which two players (called Maker and Breaker) alternately claim elements of a set $X$. Typically, in every turn Maker claims $1$ element and then Breaker claims $q$ elements, $q$ being the \textit{bias} of the game (we call the game unbiased when $q=1$). Maker wins the game if he claims all of the elements of some set from a predefined family $\mathcal{F} \subseteq \mathcal{P}(X)$, whose members are called the winning sets. If Maker failed to claim such a set, Breaker wins.
It is quite easy to verify that a Maker-Breaker game is \textit{bias monotone}. That is, if $q > q'$ and Breaker wins a Maker-Breaker game $\mathcal{F}$ with bias $q'$, then he wins the same game with bias $q$. Following this property it is natural to define the \textit{critical bias} of a Maker-Breaker game, denoted by $q^*$, to be the minimal $q$ such that Breaker wins the game with bias $q$. A general question naturally arises: what is the critical bias $q^*$ for some Maker-Breaker game? This is the main question which has been studied for many Maker-Breaker games.

In this paper, we study \textbf{Waiter-Client} games. These games, quite popular recently (see \cite{Weight functions in WC}, \cite{WC connectivity}, \cite{Manipulative W}, \cite{WC fixed graph}, \cite{WC Planarity}, \cite{WC Hamiltonicity}), were first introduced by Beck \cite{Beck Second Moment} under the name ``Picker-Chooser", and are defined as follows. The game is played by two players (Waiter and Client), and is defined by a set $X$ of elements, a family $\mathcal{F} \subseteq \mathcal{P}(X)$ of winning sets, and a positive integer
 $q$ called the bias of the game. In every turn, Waiter offers Client $q+1$ elements of $X$ which have not been previously offered. Client claims one of these elements, and Waiter claims the rest. When there are less than $q+1$ elements which were not previously offered, Waiter claims all of them and the game ends.
Waiter wins if, at the end of the game, Client claimed a whole set in $\mathcal{F}$. Client wins otherwise. Much like Maker-Breaker games, Waiter-Client games are \textit{bias monotone}. That is, if $q > q'$ and Waiter wins some Waiter-Client game with bias $q$, then he wins the same game with bias $q'$. Based on this property it is natural to define the \textit{critical bias} of a Waiter-Client game, denoted by $q^*$, to be the minimal $q$ such that Client wins the game with bias $q$. A general question naturally arises: what is the critical bias $q^*$ for some Waiter-Client game?

Note the crucial difference in the goal in Waiter-Client games as opposed to Maker-Breaker games: In a Waiter-Client game, Waiter tries to \textbf{force} Client to claim a whole set while Client tries to \textbf{avoid} it. A similar victory condition can be formalized in the Maker-Breaker setting where the players just claim elements alternatively, and such games are known as the \textbf{Avoider-Enforcer} games.

Another known type of game we would like to mention is \textbf{Client-Waiter} games. In a Client-Waiter game, the rules are the same as the Waiter-Client game, except that Client wins if and only if he claimed a winning set. This is a much more natural analogue to Maker-Breaker games. However, Client-Waiter games are not bias monotone, so the general question of finding a critical bias cannot be defined. There is a natural alteration to the rules (called the ``monotone rules") which makes the game bias-monotone: allowing Waiter to offer any number of elements between $1$ and $q+1$. In this paper, when we mention Client-Waiter games, we will speak about the first kind of rules (called ``strict", as opposed to ``monotone").

Though not always true (see \cite{Knox} for an example), Waiter is often considered a more powerful version of Breaker (or Enforcer), due to his ability to limit Client's options. This motivates us, when facing a ``hard" Maker-Breaker problem, to analyze the Waiter-Client (or Client-Waiter) version of it. Doing so might provide important insights for studying the corresponding Maker-Breaker game.

We define the \textit{Waiter-Client maximum degree $D$} game to be the Waiter-Client game played on the edges of $K_n$ with bias $q$, where Client wins if and only if all of the vertices in his graph have degree smaller than $D$. We denote this game by $WC_\Delta(n, q, D)$. Similarly, $MB_\Delta(n, q, D)$ is the Maker-Breaker game played on the edges of $K_n$ with bias $q$ in which Maker wins if and only if he achieves a vertex with degree at least $D$. When $q = 1$, we will omit $q$ and just write $WC_\Delta(n, D)$ or $MB_\Delta(n,D)$.
A natural question to ask is what is the minimum $D$ such that Client (respectively, Breaker) wins the unbiased maximum degree game. This question for Maker-Breaker has been studied for a long time. The correct asymptotic value of this critical $D$ is known and is quite easily shown to be $n/2$. The ``real" question is then finding the correct error term, or the deviation from $n/2$. Using standard tools, one could prove that Breaker wins when $D > \frac{n}{2} + C\sqrt{n \ln n}$ for some constant $C > 0$. (This also can be derived from \cite{Discrepancy Games}.) The lower bound proved to be more difficult, with Beck \cite{Beck MB Lower} providing the current best known lower bound of $\frac{n}{2} + c\sqrt{n}$ for some constant $c > 0$. The $\sqrt{\ln n}$ gap between the magnitudes of deviation from $\frac{n}{2}$ has not been closed (or even improved) since 1993.

In this paper we study $WC_\Delta(n,q,D)$ quite thoroughly. After a short review of some of the tools we use (Section \ref{preliminaries_section}), we deal with the unbiased version of the game in Section \ref{unbiased_section}, proving the correct order of magnitude of the error term of the critical $D$ of $WC_\Delta(n, D)$, as given by the following theorems:

\begin{restatable}[Upper Bound]{theorem}{upbound}\label{max_deg_upper}
Client has a strategy to win $WC_\Delta(n,\frac{n}{2} + (1+o(1))\sqrt{\frac{n \ln n}{2}})$.
\end{restatable} 

\begin{restatable}[Lower Bound]{theorem}{lowbound}\label{max_deg_lower}
Waiter has a strategy to win $WC_\Delta(n,\frac{n}{2} + \frac{\sqrt{n\ln n}}{500})$ for large enough $n$.
\end{restatable}

It should be noted that the deviation from $\frac{n}{2}$ by $\Theta(\sqrt{n \ln n})$ is (currently) a possibility for the correct magnitude of the deviation (from $\frac{n}{2}$) of the minimal $D$ such that Breaker wins $MB_\Delta(n ,D)$, and also it is the correct magnitude for the deviation (from $\frac{n}{2}$) of the maximum degree in the binomial random graph $G(n, \frac{1}{2})$, and the random graph $G(n, m)$ for $m=\left\lfloor \frac{{n \choose 2}}{2} \right\rfloor$ (which is what we get if Waiter and Client (or Maker and Breaker) play randomly). This is an example of a broad and important phenomenon, usually referred to as the \textit{probabilistic intuition} (see for example \cite{Beck Inevitable Randomness}), where the game outcome achieved by perfect play is the same as the one achieved by random play. In fact, it is known (see for example \cite{Introduction to Random Graphs}) that $\Delta(G(n,\frac{1}{2}))$ and $\Delta \left( G \left( n, \left \lfloor \frac{{n \choose 2}}{2} \right \rfloor \right) \right)$ are with high probability $\frac{n}{2} + (1+o(1))\sqrt{\frac{n \ln n}{2}}$, matching the constant of Theorem \ref{max_deg_upper}. Since our analysis for the lower bound strategy is far from achieving the optimal constants, we wonder whether the correct value for the critical $D$ is exactly $\frac{n}{2} + (1+o(1))\sqrt{\frac{n \ln n}{2}}$, precisely matching the random graph (or random play) behavior.

In an unbiased Maker-Breaker game, Maker tries to achieve some property in his graph, which is essentially equivalent to preventing a corresponding property in Breaker's graph. Thus, Maker can actually think of himself as Breaker in an auxiliary game. In the case of the Maximum Degree game, Maker can try to prevent Breaker from having a large minimum degree in his graph, which translates to a large maximum degree in Maker's graph. It is a well known approach, sometimes called \textit{building via blocking} (for example in \cite{Beck Monograph}). In every round of an unbiased Waiter-Client game, Client can choose an edge for Waiter's graph rather than for his own graph (it is virtually the same). Therefore, if Client can avoid some property in his graph (winning the Waiter-Client game) then he can avoid a similar property in Waiter's graph instead, thus achieving some property in his own graph and winning a corresponding Client-Waiter game. It is a similar property (rather than exactly the same) due to the technicality that if the number of edges is odd, Waiter takes the last edge regardless of the game. Applying these ideas to the Maximum Degree games, one can see that the Maker-Breaker (Waiter-Client) Maximum Degree game is actually equivalent to the Maker-Breaker (Client-Waiter) Minimum Degree game in a very strong sense. Therefore our results do not only draw parallels between the Waiter-Client and Maker-Breaker Maximum Degree games, but also between the Client-Waiter and Maker-Breaker Minimum Degree games, which one might deem more natural (since Client tries to achieve a property rather than to avoid it).

In Section \ref{biased_section} we deal with the general biased case. For this purpose, let us define a function $f(n,q)$ as follows: $f(n,q) = \sup \{d | (d(q+1))^d < n^{d+1} \}$. We achieve tight asymptotic results for the critical $D$ for ``almost all" possible values of $q$, as given by the following three theorems:

\begin{restatable}{theorem}{largebiasinf}\label{large_bias_inf}
Suppose $q = \omega \left(\frac{n}{\ln n}\right)$ and also $q = n^{1+o(1)}$. Playing a Waiter-Client game on the edges of $K_n$ with bias $q$, Waiter has a strategy to force Client to have a vertex with degree at least $(1+o(1))f(n,q)$, while Client has a strategy to ensure all of his vertex degrees are at most $(1+o(1))f(n,q)$.
\end{restatable}

Perhaps one of the most natural cases studied for Waiter-Client games is when $q = \Theta(n)$. Theorem \ref{large_bias_inf} solves this case in particular, and yields that for $q = cn$ for some $c > 0$ the maximum degree of Client's graph in the outcome of the Waiter-Client Maximum Degree game achieved by perfect play is $(1+o(1))\frac{\ln n}{\ln \ln n}$.

The next theorem addresses the case of a very large bias $q(n)$, much larger than $\Theta(n)$:

\begin{restatable}{theorem}{largebiasconst}\label{large_bias_const}
Let $d \geq 1$ be a natural number. There are constants $C_d, c_d > 0$ such that:

\begin{itemize}
\item Playing with bias $q \geq C_d n^{1 + \frac{1}{d}}$, Client has a strategy to ensure all of his vertices' degrees are at most $d - 1$.
\item Playing with bias $q \leq c_d n^{1 + \frac{1}{d}}$, Waiter has a strategy to force Client to have a vertex with degree at least $d$.
\end{itemize}
\end{restatable}

Finally, we present a theorem that deals with relatively small bias $q(n)$:

\begin{restatable}{theorem}{smallbias}\label{small_bias_maximum_degree}
Let $q = o \left(\frac{n}{\ln n}\right)$. Playing a Waiter-Client game on the edges of $K_n$ with bias $q$, Waiter has a strategy to force Client to have a vertex with degree at least $(1+o(1))\frac{n}{q+1}$, while Client has a strategy to ensure all of his vertices' degrees are at most $(1+o(1))\frac{n}{q+1}$.
\end{restatable}

These theorems indicate that the biased Waiter-Client Maximum Degree game is another successful example of the probabilistic intuition, as its perfect play outcome behaves precisely like its random play outcome. We expand on this matter in Section \ref{biased_section}. 

We also prove a result for the most significant case not covered by the previous theorems:

\begin{restatable}{proposition}{mediumbias}
Suppose $q = c\frac{n}{\ln n}$ for some constant $c > 0$. Then Waiter can force Client to have maximum degree at least $(1+o(1))\frac{n}{q+1}$, while Client can ensure his maximum degree is at most $C\frac{n}{q+1}$ for some $C > 0$ (which depends only on $c$).
\end{restatable}

\section{Preliminaries} \label{preliminaries_section}

We use standard graph-theoretic notation, in particular $\Delta(G)$ is the maximum degree of a graph $G$ and $\delta(G)$ is the minimum degree of a graph $G$. We may write $\Delta, \delta$ when the graph $G$ is clear from the context.

For any Waiter-Client game on the edges of $K_n$:
\\ Let $E^F_i, E^C_i, E^W_i$ denote the free edges, Client's edges and Waiter's edges after round $i$, respectively. 
\\ Let $d_C^i(v)$ ($d_W^i(v)$) denote the degree of $v$ in Client's (Waiter's) graph after round $i$. Often we will omit $i$ when it is obvious from the context.
\\ Let $G_C^i$ ($G_W^i$) be the graph consisting of Client's (Waiter's) edges after round $i$. Often we will omit $i$ when it is obvious from the context.

We will use the following general and quite useful result regarding martingales:

\begin{theorem}[Azuma's Inequality, see, e.g.,  \cite{The Probabilistic Method}]
Let $0 = X_0, \dots, X_m$ be a martingale which satisfies 
\[|X_{i+1} - X_i| \leq 1\]
for all $0 \leq i < m$. Let $\lambda > 0$ be arbitrary. Then 
\[\Pr(X_m > \lambda\sqrt{m}) < e^{-\lambda^2/2}. \]
\end{theorem}

We now state a result of Beck, which adapts the known Erd\H{o}s Selfridge  general criteria for Breaker's win in a Maker-Breaker game to a Waiter-Client game:

\begin{theorem}[Implicit in \cite{Beck Monograph}]\label{client_general_criteria}
Let $q$ be a positive integer, let $X$ be a finite set, let $\mathcal{F}$ be a family of (not necessarily distinct) subsets of $X$ and suppose $\sum_{A \in \mathcal{F}} (q+1)^{-|A|} < 1$. Then, when playing the $(1 : q)$ Waiter-Client game $(X, \mathcal{F})$, Client has a winning strategy.
\end{theorem}

For a short explicit proof of Theorem \ref{client_general_criteria} one can look in \cite{Waiter-Client Planarity}.

\section{Unbiased Case} \label{unbiased_section}

In this section we will prove Theorem \ref{max_deg_upper} and Theorem \ref{max_deg_lower}.

\subsection{Upper Bound}

In this subsection we will prove a generalization of Theorem \ref{max_deg_upper}. We will work in a general Unbalancer-Balancer framework (adapted to the Waiter-Client rules), and present a strategy which is applicable for many games. We rely heavily on ideas from \cite{Discrepancy Games}, adapted to Waiter-Client style games.
First, we define what is such a game:

Given a hypergraph $H = (V,E)$ with $E = \{e_1, \dots, e_m\}$ and two threshold vectors $\textbf{l} = (l_1, \dots, l_m)$ and $\textbf{h} = (h_1, \dots, h_m)$ of non-negative integers, we define the \textit{Unbalancer-Balancer game} on $H$ with thresholds $\textbf{l}, \textbf{h}$, denoted by $UB(H, \textbf{l}, \textbf{h})$, as follows: in every round, Unbalancer offers two elements $v_1, v_2 \in V$ which have not already been claimed by any player. Balancer then chooses one of them, claims it, and Unbalancer claims the other element. The game ends when there are less than 2 unclaimed elements (if there is one such element, it stays unclaimed). Let $B$ and $U$ be the subsets claimed by Balancer and Unbalancer, respectively, at the end of the game. If for every edge $e_j \in E$ it holds that $-l_j \leq |B \cap e_j| - |U \cap e_j| \leq h_j$, Balancer wins. Otherwise, Unbalancer wins.

It is appropriate to note that in contrast to \cite{Discrepancy Games}, which defined the \textit{Balancer-Unbalancer} games analogously to Maker-Breaker games in some sense, we defined it differently and analogously to Waiter-Client games in some sense. Also, we properly changed the order of the names, calling it an \textit{Unbalancer-Balancer} game, as Unbalancer resembles Waiter, and Balancer resembles Client.

Another (equivalent) way of looking at the game is by using a labeling function $f: V\rightarrow \{-1,0,1\}$. At the start of the game, we let $f(v) = 0$ for all $v \in V$. When Balancer claims an element $v$, we update $f(v) = 1$. Similarly, when Unbalancer claims an element $u$, we update $f(u) = -1$. At the end of the game, Balancer wins if and only if for every $e_j \in E$, it holds that $-l_j \leq \sum_{v \in e_j} f(v) \leq h_j$.

Informally speaking, Balancer's goal is to balance every edge $e \in E$, making sure that both players have about the same number of elements claimed in that edge. The exact meaning of ``about the same number" is given by $\textbf{l}, \textbf{h}$. The link to the Waiter-Client maximum degree game is straightforward, and will be formalized when proving Theorem \ref{max_deg_upper}.

Our main result for Unbalancer-Balancer games is the following:

\begin{theorem}\label{general_UB}
Let $H = (V,E)$ be a hypergraph with $E = \{e_1, \dots, e_m\}$, and let $\textbf{l} = (l_1, \dots, l_m), \textbf{h} = (h_1, \dots, h_m)$  be two threshold vectors of non-negative integers. Then Balancer has a winning strategy for $UB(H,\textbf{l}, \textbf{h})$, if 
\[\sum_{j=1}^m e^{-\frac{h_j^2}{2|e_j|}} + e^{-\frac{l_j^2}{2|e_j|}} < \frac{1}{2}.\]
In particular, Balancer has a winning strategy if $l_j,h_j > \sqrt{2 \ln (4m)|e_j|}$ for every $j$.
\end{theorem}

To prove Theorem \ref{general_UB}, we will use a potential (or a weight) function. Our potential function will use the expressions $P_{k,M}$, which we define as follows:
Let $X_1, X_2, \dots, X_k$ be $k$ i.i.d. random variables, uniform on $\{-1, 1\}$. Let $S_j = \sum_{i=1}^j X_i$ for $1 \leq j \leq k$ be the random walk generated by $X_i$. Let also $S_0 = 0$. We define $P_{k,M} = \Pr (\max_{j\geq 0} S_j > M)$, the probability that at some point during the random walk its value was more than $M$. Note that $P_{k,M} = 1$ for $M < 0$, and also note the equality $P_{k, M} = \frac{1}{2}P_{k - 1, M - 1} + \frac{1}{2}P_{k - 1, M + 1}$ when $M \geq 0$, which is trivial by the definition.

Suppose Unbalancer and Balancer are playing $UB(H, \textbf{l}, \textbf{h})$, labelling the claimed vertices by $-1$ and $1$, respectively. At any point in the game, for $e_j \in E$, we denote $k_j = |\{v \in e_j | f(v) = 0\}|$ and $d_j = \sum_{v \in e_j} f(v)$. We define the \textit{positive weight} of $e_j$, denoted by $W_j^+$, to be $W_j^+ = P_{k_j, h_j - d_j}$. Similarly we define the \textit{negative weight} of $e_j$ to be $W_j^- = P_{k_j, d_j + l_j}$. The \textit{total weight} of the game is $W = \sum_j (W_j^+ + W_j^-)$.

\begin{lemma}
If $\sum_{j}P_{|e_j|, h_j} + \sum_{j}P_{|e_j|, l_j} < 1$, Balancer has a strategy to win.
\end{lemma}

\begin{proof}
Note that the condition $\sum_{j}P_{|e_j|, h_j} + \sum_{j}P_{|e_j|, l_j} < 1$ is, by definition, the requirement that $W < 1$ at the start of the game. Balancer's strategy is very simple to describe: when Unbalancer offers elements $x,y$, Balancer chooses the element which will minimize the updated weight $W$. That is, Balancer calculates the new weight $W_x$, achieved by assigning $f(x) = 1, f(y) = -1$, and similarly the new weight $W_y$, achieved by assigning $f(x) = -1$, $f(y) = 1$. If $W_x < W_y$ Balancer chooses $x$. Otherwise, he chooses $y$.

In fact we will now prove that, following this strategy, Balancer guarantees that if $W < 1$ before some round, then $W$ does not increase after that round, and in particular $W < 1$.
$W$ is the sum of $W_j^+$ and $W_j^-$ for all $j$. Fix $j$ and look at what happens to $W_j^+$ when Balancer chooses $x$ or $y$ (denoted $W_{j,x}^+$ and $W_{j,y}^+$ respectively). We begin by claiming that $W_j^+ \geq \frac{1}{2}(W_{j,x}^+ + W_y^+)$. To prove it, note that $W_j^+ < 1$ because $W < 1$, and consider the following four cases:
\\ In the first case, $x,y \notin e_j$. In this case clearly $k_j, d_j$ do not change and $W_j^+ = W_{j,x}^+ = W_{j,y}^+$ regardless of what Balancer chooses, and in particular $W_j^+ = \frac{1}{2}(W_{j,x}^+ + W_{j,y}^+)$.
\\ In the second case, $x,y \in e_j$. In this case, $k_j$ decreases by $2$ and $d_j$ does not change regardless of what Balancer does. Since clearly $P_{k-2, M} \leq P_{k, M}$ for all $k \geq 2$ and $M$, we have that $W_j^+ \geq W_{j,x}^+ = W_{j,y}^+$, and in particular $W_j^+ \geq \frac{1}{2}(W_{j,x}^+ + W_{j,y}^+)$ holds.
\\ In the third case, $x \in e_j$ and $y \notin e_j$. In this case, $k_j$ decreases by $1$ regardless of what Balancer does. If Balancer picks $x$, $d_j$ increases by $1$. Otherwise, it decreases by $1$. This means that $W_{j,x}^+ = P_{k_j - 1, h_j - d_j - 1}$ and $W_{j,y}^+ = P_{k_j - 1, h_j - d_j + 1}$. As mentioned previously, $P_{k_j, h_j - d_j} = \frac{1}{2}P_{k_j - 1, h_j - d_j - 1} + \frac{1}{2}P_{k_j - 1, h_j - d_j + 1}$ when $h_j - d_j \geq 0$, which is true in our case because $W_j^+ < 1$. Hence we get $W_j^+ = \frac{1}{2}(W_{j,x}^+ +W_{j,y}^+)$.
\\ The fourth case, in which $x \notin e_j$ and $y \in e_j$, is analogous to the third case, and the same conclusion applies: $W_j^+ = \frac{1}{2}(W_{j,x}^+ + W_{j,y}^+)$.
\\This proves the claim that $W_j^+ \geq \frac{1}{2}(W_{j,x}^+ + W_y^+)$. Clearly, the same analysis holds when observing $W_j^-$ instead, yielding $W_j^- \geq \frac{1}{2}(W_{j,x}^- + W_{j,y}^-)$.

From the analysis for each $W_j^+$ and $W_j^-$ it follows that 
\[W = \sum_j(W_j^+ + W_j^-) \geq \frac{1}{2}\sum_j(W_{j,x}^+ + W_{j,y}^+ + W_{j,x}^- + W_{j,y}^-) = \frac{1}{2}(W_x + W_y).\]
In particular, either $W_x$ or $W_y$ must be at most as large as $W$, hence Balancer's strategy guarantees the new weight to be at most $W$, which is smaller than $1$.

When the game ends, clearly still $W < 1$. Now we need to show that for any edge $e_j$, it holds that $-l_j \leq d_j \leq h_j$. This is immediate from the fact that $W_j^+$ and $W_j^-$ are smaller than $1$. Indeed, if $d_j > h_j$, then we have $W_j^+ = P_{k_j, h_j - d_j} = 1$, a contradiction. Similarly if $d_j < -l_j$ then we have $W_j^- = P_{k_j, d_j + l_j} = 1$, a contradiction. Therefore Balancer won the game.
\end{proof}

Now proving Theorem \ref{general_UB} is quite easy --- one only needs to remember that $P_{k,M} \leq 2e^{\frac{-M^2}{2k}}$, and to use the previous theorem. The latter inequality is simply proven using the mirror principle and the Chernoff bound.

We can now use this general result to deduce Theorem \ref{max_deg_upper} for the Waiter-Client Maximum-Degree game:

\begin{proof}[Proof of Theorem \ref{max_deg_upper}] Client views himself as a Balancer in a Unbalancer-Balancer game on the hypergraph $H$ whose (hyper-)vertices are the edges of $K_n$. The (hyper-)edges in $H$ are $e_v = \{(v,u) \in E(K_n) | u \in V(K_n)\}$, all the edges of $K_n$ touching $v$. Thus $H$ has $n$ edges, all of size $n-1$. Due to Theorem \ref{general_UB}, Client has a strategy which guarantees that at the end of the Unbalancer-Balancer game, he has at most $(1+o(1))\sqrt{2n \ln n}$ more elements than Waiter in every hyper-edge $e_v$. It follows that for every $v$, Client could not have taken more than $\frac{n}{2} + (1+o(1))\sqrt{\frac{n \ln n}{2}}$ edges touching $v$. When the Unbalancer-Balancer game ends, the Waiter-Client game essentially ends as well (Waiter might claim one remaining edge of $K_n$). This shows that Client successfully avoided degree at least $\frac{n}{2} + (1+o(1))\sqrt{\frac{n \ln n}{2}}$ for every vertex $v$, and thus won the game.
\end{proof}

\subsection{Lower Bound}

In this subsection we will prove Theorem \ref{max_deg_lower} and thus give a lower bound for the maximum degree Waiter can force Client to have in the unbiased game. We start by describing the concept of ``Advantage". We then describe and analyze a certain combinatorial process, and finally utilize this process to give an explicit strategy for Waiter.
\subsubsection{Advantage}
Suppose Waiter and Client are playing on the edges of a graph $G = (V,E)$. Define the following notation: $E_v = \{e \in E | v \in e\}$, $E_i^W(v) = E_i^W \cap E_v$, and similarly $E_i^C(v) = E_i^C \cap E_v$. We define the \textit{Advantage} at a vertex $v$ (after round $i$) to be $|E_i^C(v)| - |E_i^W(v)|$, and denote it by $A_i(v)$. We will also use a more generalized concept of advantage, the advantage between $v$ and a set $U$ of vertices. We define $A_i^U(v) = |\{u \in U | (v,u) \in E_i^C\}| - |\{u \in U | (v,u) \in E_i^W\}|$, the advantage of $v$ in $U$. When we assume it is clear from the context, we may omit the subscript $i$.

The advantage of a vertex is a crucial number when thinking about the maximum degree game, because Waiter can keep the accumulated advantage at a vertex for the rest of the game. The formal statement is given by the following simple claim:

\begin{claim}\label{advantage_win}
Suppose that after $i$ rounds of an unbiased Waiter-Client game on $E(K_n)$ there is a vertex $v \in V$ such that $A_i(v) = a$. Then Waiter has a strategy to ensure that at the end of the game $d^C(v) \geq \frac{n}{2} + \frac{a}{2} - 1 $.
\end{claim}

\begin{proof}
Let $E_i^F = \{(v_1, v), (v_2,v), \dots, (v_k,v)\}$ be the free edges touching $v$ after round $i$. For the next $m = \left \lfloor \frac{k}{2} \right \rfloor$ rounds, Waiter offers $(v_{2j-1},v)$ and $(v_{2j}, v)$ during round $i + j$. Clearly the advantage does not change during these rounds, as both Waiter and Client each get one edge touching $v$ every round. Therefore, $A_{i+ m} = A_i(v) = a$, thus $d^C_{i+m}(v) - d^W_{i+m} \geq a$. Also, at least $n-2$ of the edges in $E_v$ have been offered, therefore $d^C_{i+m}(v) + d^W_{i+m} \geq n-2$. This yields $d^C_{i+m}(v) \geq \frac{n}{2} + \frac{a}{2} - 1 $. Since $d^C(v)$ cannot decrease during the game, we conclude the desired result.
\end{proof}
\subsubsection{Advantage Process}
The \textit{Advantage Process} is a simple deterministic process. It works as follows: suppose we have an infinite number of boxes labeled with the integers, and suppose we start with $b$ balls in box number $0$ (all the other boxes are empty). The process works in ``rounds". Let $a_{j,i}$ denote the number of balls in box $i$ after round $j$. Thus, $a_{0,i} = 0$ for all $i \neq 0$, and $a_{0,0} = b$. In round $j$, for every box $i$ we take half (rounded down) of its balls and move them to box $i+1$ (we say these balls ``move right"), and half (rounded down) of its balls and move them to box $i-1$ (we say these balls ``move left"). If box $i$ had an odd number of balls in it, then one of them remains inside box $i$. Formally: 
\begin{equation*}
a_{j,i} = \left \lfloor \frac{a_{j-1,i-1}}{2} \right \rfloor + \left \lfloor \frac{a_{j-1,i+1}}{2} \right \rfloor + (a_{j-1,i}\, mod  \ 2).
\end{equation*}
The Advantage Process gets two parameters; $b$, the number of balls, and $r$, the number of rounds. The quantities $a_{j,i} = a_{j,i}(b,r)$ can then be calculated, and we denote by $a_i^* = a_{r,i}$ the number of balls in box $i$ after $r$ rounds, when starting the process with $b$ balls in box $0$. In the following text, we will often use $a_i^*$ without mentioning $b$ and $r$, when they are clear from the context.

We next analyze the Advantage Process, aiming to prove the following Lemma:

\begin{restatable}{lemma}{manyfar}\label{many_far}
Suppose $b \geq r^{\frac{2}{3}}$ and $r$ is large enough. Then after running the Advantage Process with parameters $b,r$, at least $0.03 b$ balls are at a box at least as large as $0.1\sqrt{r}$.
\end{restatable}
This lemma will have a key role in devising Waiter's strategy.

Note that the Advantage Process is a deterministic process. However, for the next couple of claims, we will adapt the following ``probabilistic" view of the process: Label the balls $1$ to $b$, and  during every round, for every box the balls that ``move right" and the balls that ``move left" are chosen randomly in every box (such that half of them (rounded down) move left and the same amount move right). We can then consider the random variable $X_k$, $1 \leq k \leq b$, to be the box in which ball $k$ is at the end of the process. We also denote $X_k^j$ to be the box in which ball $k$ is after round $j$. Note that since the amounts of balls in each box follow the above described deterministic process, $|\{k | X_k = i\}| = a_i^*$. By symmetry, $X_k$ all have the same distribution, and therefore for all $i,k$ it holds that $\Pr (X_k = i) = a_i^* / b$.

We now prove some properties of the Advantage Process:

\begin{claim}\label{symmetry}
The Advantage Process is symmetric, that is, $a_{j,i} = a_{j,-i}$ for all $i$.
\end{claim}

\begin{proof}
We prove this by induction on $j$. By definition, for $j = 0$, $a_{j,i} = 0$ for all $i \neq 0$, so symmetry holds. Suppose it is true for $j-1$ (and all $i$). For any $i$, we have 
\[ a_{j,i} = \left \lfloor \frac{a_{j-1,i-1}}{2} \right \rfloor + \left \lfloor \frac{a_{j-1,i+1}}{2} \right \rfloor + (a_{j-1,i}\, mod  \ 2)  = \left \lfloor \frac{a_{j-1,-i+1}}{2} \right \rfloor + \left \lfloor \frac{a_{j-1,-i-1}}{2} \right \rfloor + (a_{j-1,-i}\, mod  \ 2) = a_{j,-i},\] where the second equality is due to the inductive assumption.
\end{proof}

Note that Claim \ref{symmetry} implies that $X_k$ has a symmetric distribution (around 0).

\begin{proposition}\label{tail_bound}
$\Pr \left(X_k \geq t\sqrt{r}\right) \leq e^{\frac{-t^2}{2}}$ for all $t > 0$.
\end{proposition}

\begin{proof}
Note that $\mathbb{E}\left(X_k^j | X_k^{j-1}\right) = X_k^{j-1}$. Indeed, suppose $X_k^{j-1} = z$. Denoting $s = a_{j-1,z}$, we know the exact probabilities $\Pr \left(X_k^j = z + 1|X_k^{j-1} = z\right) = \Pr \left(X_k^j = z - 1|X_k^{j-1} = z\right) = \frac{\left \lfloor \frac{s}{2} \right \rfloor}{s}$. If $s$ is odd, there is also a positive probability that $X_k^j = z$. Anyway, it is clear that $\mathbb{E}\left(X_k^j | X_k^{j-1} = z\right) = z$. Hence the sequence $\left(X_k^j\right)_{j\geq 0}$ is a martingale which cannot change by more than $1$ (in absolute value). Using Azuma's inequality, we have $\Pr \left(X_k \geq t\sqrt{r}\right) = \Pr \left(X_k^r \geq t\sqrt{r}\right) \leq e^{-\frac{t^2}{2}}$.
\end{proof}

\begin{corollary}\label{not_far}
$a_i^* = 0$ for all $i$ such that $|i| > \sqrt{2r\ln b}$.
\end{corollary}

\begin{proof}
Take $t$ such that $t\sqrt{r} = \lceil \sqrt{2r\ln b} \rceil$. Clearly $t \geq \sqrt{2\ln b}$. By the previous lemma and the fact that $\Pr \left(X_k = i\right) = \frac{a_i^*}{b}$, we have $\sum_{i=t\sqrt{r}}^\infty a_i^* < be^{-\frac{2\ln b}{2}} = 1$. Since $a_i^*$ are all non-negative integers, we conclude the corollary for positive $i$. By symmetry of the process, we conclude it for negative $i$.
\end{proof}

\begin{lemma}\label{potential}
Suppose $b \geq r^{\frac{2}{3}}$. Then $\sum_{k=1}^b (X_k)^2 = (1-o(1))br$. (Here, $o(1)$ means when $r \rightarrow \infty)$
\end{lemma}

\begin{proof}
Let us define $S_j = \sum_{k=1}^b (X_k^j)^2$. Clearly, $S_0 = 0$. Also let us define, for every round $j$ of the (randomized) Advantage Process, a bijection pairing $f_j: [b] \rightarrow [b]$ such that $f_j \circ f_j = Id_{[b]}$, and also for any ball $k$ the following holds: if ball $k$ did not ``move" during round $j$, $f_j(k) = k$. If ball $k$ moved during round $j$ from box $i$ to box $i+1$, then $f_j(k)$ is a ball that moved during round $j$ from box $i$ to box $i-1$. Similarly, if ball $k$ moved from box $i$ to box $i-1$, $f_j(k)$ is a ball that moved from $i$ to $i+1$.
Clearly such pairings exist, since the same number of balls move from $i$ to $i+1$ as the number of balls that move from $i$ to $i-1$. 

Now note that for a ball $k$, that moved from box $i$ during round $j$, we have \[\left(X_k^j\right)^2 + \left(X_{f_j(k)}^j\right)^2 - \left(\left(X_k^{j-1}\right)^2 + \left(X_{f_j(k)}^{j-1}\right)^2\right) = (i+1)^2 + (i-1)^2 - 2i^2 = 2.\]
In other words, every pair of balls in a box that moved from it (one increased and one decreased) contributed $2$ to the sum of $\left(X_k^j \right)^2$. It is also clear that for a ball $k$ that did not move, it holds that $\left(X_k^j\right)^2 = \left(X_k^{j-1}\right)^2$. Therefore, $S_j - S_{j-1}$ is precisely the number of balls that moved during round $j$. Due to Corollary \ref{not_far}, we know that the number of boxes with balls in it is at most $2\sqrt{2r\ln b} = o(b)$ (by the assumption on $b$). In other words, $(1-o(1))b$ balls move every round, and therefore $(1-o(1))b \leq S_j - S_{j-1} \leq b$ and thus $S_r = (1-o(1))br$.
\end{proof}

We can now prove Lemma \ref{many_far}, which is our goal for this subsection. Let us repeat its statement:

\manyfar*

\begin{proof}
Denote $h = \lfloor 0.1\sqrt{r} \rfloor, H = \lfloor \sqrt{5r} \rfloor$. Due to Lemma \ref{potential} and the fact that every $X_k$ has the same distribution, we know that $\mathbb{E}(X_1^2) = (1+o(1))r$. Also, due to Claim \ref{symmetry},  \[ \mathbb{E}(X_1^2) = \sum_{i=-\infty}^\infty i^2 \Pr (X_1=i) = 2\sum_{i=1}^\infty i^2 \Pr (X_1=i). \] We have \[ \sum_{i=1}^\infty i^2 \Pr (X_1=i) = \sum_{i=1}^{h} i^2 \Pr (X_1=i) + \sum_{i=h+1}^{H - 1} i^2 \Pr (X_1=i) + \sum_{i=H}^\infty i^2 \Pr (X_1=i). \] We will analyze each summand separately. The first two are simple:
\vspace{-.2cm}
\begin{align*}
\sum_{i=1}^{h} i^2 \Pr (X_1=i)& \leq h^2  \leq 0.01r,\\
\sum_{i=h+1}^{H - 1} i^2 \Pr (X_1=i)& \leq H^2 \Pr (X_1 \geq h+1) \leq 5r \Pr (X\geq 0.1\sqrt{r})
\end{align*}

For the last summand, by changing the order of summation, we have
\begin{equation*}
\begin{split}
\sum_{i=H}^\infty i^2 \Pr (X_1=i) & = H^2 \Pr (X_1\geq H) + \sum_{i=H+1}^\infty \left(i^2 - (i-1)^2\right) \Pr (X_1\geq i) \\
& \leq 5r \Pr (X_1 \geq 0.1\sqrt{r}) + \sum_{i=H+1}^\infty 2i \Pr (X_1\geq i).
\end{split}
\end{equation*}
To complete the analysis, we use Proposition \ref{tail_bound} and get that \[ \sum_{i=H+1}^\infty 2i \Pr (X_1\geq i) \leq \sum_{i=H+1}^\infty 2i e^{-\frac{i^2}{2r}} = (1+o(1))\int_{H+1}^{\infty}2xe^{-\frac{x^2}{2r}}dx = (1+o(1))2re^{-\frac{5}{2}}.\]
\\Finally, we can put it all together. We have: 
\[(1+o(1))\frac{r}{2} \leq 0.01r + 10r \Pr (X_1 \geq 0.1\sqrt{r}) + (1+o(1))2re^{-\frac{5}{2}},\] yielding
\[ \Pr (X_1 \geq 0.1\sqrt{r}) \geq \frac{1}{10}\left(\frac{1}{2} - 0.01 - (1+o(1))2e^{-\frac{5}{2}}\right) > 0.03\] for large enough $r$. Now recall the fact that $a_i^* = b \Pr (X_1 = i)$ for all $i$ to derive the statement.
\end{proof}

Note that in the above proof we did not optimize the constants. We were not able to match the constant from the upper bound, so we chose not to pursue the optimal constants possible for the sake of simplicity.

\subsubsection{Reduction}

In this subsection we wish to present Waiter's strategy. The main point we use in the strategy is Waiter's ability to force Client to simulate the Advantage Process. By doing so, Waiter affects the advantage of vertices in the graph, and causes them to behave like in the advantage process. Using our previous analysis of the process, Waiter can indeed reach a point where a vertex has an advantage of $c\sqrt{n\ln n}$, for some $c > 0$. We begin by proving the following lemma, which formalizes Waiter's ability to simulate the Advantage Process.

\begin{lemma} \label{single_stage}
Suppose Waiter and Client play an unbiased Waiter-Client game on $E(K_{B,R})$ for some sets $B, R$. Let $b=|B|, r=|R|$, and $\{a_i^*\}_{i=-\infty}^\infty$ be the result of the Advantage Process with parameters $b,r$. Then Waiter has a strategy to force the following: at some point in the game, for all $i$, exactly $a_i^*$ vertices in $B$ have advantage $i$.
\end{lemma}

\begin{proof}
Waiter's strategy consists of $r$ phases, each phase corresponding to a ``round" in the Advantage Process. Suppose $B = \{v_1, v_2, \dots, v_b\}$ and $R = \{u_1, u_2, \dots, u_r\}$, and let $a_{j,i}$ be as defined in the Advantage Process (with parameters $b, r$). We prove that Waiter can ensure that after phase $j$, these two conditions hold:

\begin{itemize}
\item For all $i$, exactly $a_{j,i}$ vertices in $B$ have advantage $i$.
\item For all $j' > j$, all edges touching $u_{j'}$ are still free and have not been offered.
\end{itemize}

For the induction basis, it is obvious that after phase $0$ (that is, right in the beginning of the game) both conditions hold. Suppose phase $j$ has just finished, and the two conditions hold.  At this moment, define $B_{j,i} = \{v \in B | A^j(v) = i\}$. By assumption, we know $|B_{j,i}| = a_{j,i}$. We now describe Waiter's moves during the $(j+1)$-th phase: for each $i$, Waiter arbitrarily chooses distinct pairs of vertices $v, v' \in B_{j,i}$ and offers Client the edges $(v,u_{j+1}), (v', u_{j+1})$. Waiter does this until $(v,u_{j+1})$ was offered for every $v \in B_{j,i}$, except maybe for one in the case that $a_{j,i}$ is odd. Phase $j+1$ ends when Waiter completes doing this for every $i$.

All that is left to do is prove that the conditions hold. The second condition clearly holds, because during phase $j+1$ Waiter only offered edges touching $u_{j+1}$. As for the first condition, let us count how many vertices have advantage $i$ at the end of phase $j+1$: during phase $j+1$, every vertex $v \in B$ belonged to at most one edge offered, thus its advantage could not change by more than $1$. How many vertices in $B_{j,i-1}$ had their advantage increased? Exactly the number of such $v \in B_{j,i-1}$ which had $(v,u_{j+1})$ offered and taken by Client. Such edges were only offered in the $\left \lfloor \frac{a_{j,i-1}}{2} \right \rfloor$ pairs offered to Client of the form $(v,u_{j+1}), (v', u_{j+1})$ where $v, v' \in B_{j,i-1}$. Since exactly one edge of each such pair was chosen by Client, exactly  $\left\lfloor \frac{a_{j,i-1}}{2} \right\rfloor$ vertices had advantage $i - 1$ and increased their advantage to $i$.

In a similar fashion, exactly $\left \lfloor \frac{a_{j,i+1}}{2} \right \rfloor$ vertices had advantage $i+1$ and decreased their advantage to $i$ (these are the vertices whose edges were not chosen, from the pairs offered from $B_{j,i+1}$. Also, all of the vertices in $B_{j,i}$ increased or decreased their advantage by $1$, except for one vertex if $a_{j,i}$ is odd. Therefore, the number of vertices with advantage $i$ at the end of phase $j+1$ is exactly $\left \lfloor \frac{a_{j,i-1}}{2} \right \rfloor + \left \lfloor \frac{a_{j,i+1}}{2} \right \rfloor + (a_{j-1,i}\, mod  \ 2) = a_{j+1, i}$. 

Following this strategy, Waiter completes phase $r$ and the two conditions hold. In particular, exactly $a_{r,i} = a_i^*$ vertices have advantage $i$.
\end{proof}

We can now prove our goal for this subsection. We repeat it's statement:
\lowbound*

\begin{proof}
We describe an explicit strategy for Waiter. Let $s = \left\lceil\frac{\ln n}{15} \right\rceil$, and arbitrarily choose pairwise disjoint sets of vertices $R_1, R_2, \dots, R_s \subset V$, each of size $|R_i| = \left\lceil \frac{n}{\ln n} \right\rceil$. Denote $R = R_1 \cup R_2 \cup \dots \cup R_s$ and $A_1 = V \setminus R$ be the remaining vertices, and note that $|A_1| = n - s|R_1| > 0.9n$. Waiter's strategy works in $s$ stages. During the game, Waiter maintains a set of vertices disjoint from $R$ (starting with $A_1$, and updated to $A_{j+1} \subseteq A_j$ after stage $j$). During stage $j$, Waiter significantly increases the advantage of many vertices in $A_j$, playing only on the edges between $A_j$ and $R_j$. To put it formally: at the end of stage $j$, we define $A_{j+1} = \{v \in A_j | A^{R_j}(v) \geq 0.1\sqrt{\frac{n}{\ln n}}\}$, where $A^{R_j}(v)$ is the advantage of $v$ in $R_j$ (at the end of stage $j$). Waiter has a strategy to maintain the following conditions:

\begin{itemize}
\item During stage $j$, only edges between $A_j$ and $R_j$ are offered.
\item $|A_{j+1}| \geq 0.03|A_j|$.
\end{itemize}

First we explain why if Waiter maintains these conditions, he wins and we are done. We start by claiming that $A_{s+1}$ is not empty. By the second condition we know that 
\[|A_{s+1}| \geq 0.03^s|A_1| > 0.9n(\frac{1}{n})^\frac{1}{3} > 0.\]
Actually, it turns out that $A_{s+1}$ is pretty large! Take some vertex $v \in A_{s+1}$. It is easy to see that $A(v) = \sum_{j=1}^s A^{R_j}(v)$, since no edges touching $v$ were offered except edges to $R$, and since $R_j$ are pairwise disjoint. Since $v \in A_{s+1}$, we know that  $v \in A_j$ for all $j$, and this means $A^{R_j}(v) \geq 0.1\sqrt{\frac{n}{\ln n}}$. In total, $A(v) \geq 0.1s\sqrt{\frac{n}{\ln n}} \geq \frac{\sqrt{n\ln n}}{150}$. Using Claim \ref{advantage_win}, it now follow that Waiter can ensure that $d^C(v) \geq \frac{n}{2} + \frac{\sqrt{n\ln n}}{300} - 1 >  \frac{n}{2} + \frac{\sqrt{n\ln n}}{500}$ by the end of the game.

We will now prove that Waiter can maintain the above stated conditions. During stage $j$, Waiter applies Lemma \ref{single_stage}, with $B = A_{j}$ and $R = R_j$. Note that indeed, due to the first condition (held in previous stages), all of the edges between $A_{j}$ and $R_j$ are free. Also note that since Waiter applies the strategy given in Lemma \ref{single_stage}, it is ensured that the first condition holds for stage $j$ as well.
By Lemma \ref{single_stage}, we are guaranteed that exactly $a_i^*$ vertices  $v \in A_{j}$ will have $A^{R_j}(v) = i$. By Lemma \ref{many_far}, we know that, given that $b \geq r^{\frac{2}{3}}$, at least $0.03b$ balls are in a box at least as large as $0.1\sqrt{r}$. In our case, $b = |A_{j}|$ and $r = \left \lceil \frac{n}{\ln n} \right \rceil$, so we have that $|A_{j+1}| = |\{v \in A_{j} | A^{R_j}(v) \geq 0.1\sqrt{\frac{n}{\ln n}}\}| \geq 0.03|A_{j}|$. We only need to justify that indeed $b \geq r^{\frac{2}{3}}$; This is done easily by noticing that $|A_{j}| >  0.03^s|A_1| > 0.9n(\frac{1}{n})^\frac{1}{3} = 0.9n^\frac{2}{3} > \left \lceil \frac{n}{\ln n} \right \rceil^\frac{2}{3}$. This completes the proof.
\end{proof}

\section{Biased Case} \label{biased_section}

\subsection{Large Bias}
In this subsection we will analyze the Waiter-Client Maximum Degree game played with a large bias ($q=\omega \left(\frac{n}{\ln n}\right)$), proving tight results for the degree that Waiter can force, and that Client can avoid, playing $WC_\Delta(n, q, D)$.

Recall that we defined a function $f(n,q)$ as follows: $f(n,q) = \sup \{d | (d(q+1))^d < n^{d+1} \}$. It is quite clear that $f(n,q)$ is finite for all $n,q$. We prove the following results (as stated in Section \ref{introuction_section}):

\largebiasinf*

\largebiasconst* 

The first result practically solves the maximum degree game for bias much larger than $\frac{n}{\ln n}$, but still small enough to be $n^{1+o(1)}$. The second result solves the maximum degree game for larger bias, except for very specific intervals ($q \in [c_d n^{1 + \frac{1}{d}}, C_d n^{1 + \frac{1}{d}}]$).

Both results use the same simple strategies for Waiter and Client. We begin by analyzing Client's side, which amounts to applying the general criteria for Client's win.

\begin{proposition}\label{large_bias_client}
Let $d,q$ positive be integers and suppose that $(\frac{d(q+1)}{e})^d > n^{d+1}$. Then playing a Waiter-Client game on the edges of $K_n$ with bias $q$, Client has a strategy to ensure all of his vertices' degrees are at most $d-1$.
\end{proposition}

\begin{proof}
We just need to check that the general criterion given by Theorem \ref{client_general_criteria} holds. There are exactly $n {n-1 \choose d}$ winning sets, each of size exactly $d$. Thus the criteria says that Client has a strategy to win (that is, avoid any vertex with degree $d$) if $n {n-1 \choose d}(q+1)^{-d} < 1$. Indeed, we have that $n {n-1 \choose d}(q+1)^{-d} < n\left(\frac{ne}{d(q+1)}\right)^d < 1$ by assumption.
\end{proof}

The above proposition leads us to define $D_n^q := min\{d \in \mathbb{N} | \left(\frac{d(q+1)}{e}\right)^d > n^{d+1}\}$. We now know that Client can always ensure his maximum degree is less than $D_n^q$, when playing the maximum degree game with bias $q$. When we will present Waiter's strategy, a similar expression will appear. With foresight, we define now $d_n^q := max\{d \in \mathbb{N} |3\left(4d(q+1)\right)^d < n^{d+1}\}$.

\begin{lemma}\label{large_bias_waiter}
Playing a Waiter-Client game on the edges of $K_n$ with bias $q$, Waiter has a strategy to force Client to have a vertex with degree at least $d_n^q$.
\end{lemma}

\begin{proof}
For the simplicity and readability of the proof, we omit the subscript and superscript of $d_n^q$, and for this proof we use $d$ instead. We just need to remember that $3(4d(q+1))^d < n^{d+1}$. First, note that if $d(q+1) \leq n$, then the average degree in Client's graph is larger than $d-1$. Looking at the definition of $d_n^q$, it is clear that $d < \frac{n}{2}$ for any possible $q$, and so we can bound the average degree in Client's graph:
\[ \bar{d}(G_C) = \frac{2}{n} \left\lfloor \frac{n(n-1)}{2(q+1)} \right\rfloor \geq  \frac{2}{n} \left\lfloor \frac{d(n-1)}{2} \right\rfloor  \geq \frac{dn - d - 1}{n} > \frac{dn - n}{n} = d-1\].

In this case Waiter wins no matter how he plays. Therefore, we assume that $d(q+1) > n$. Let $r = \left \lceil \frac{3d(q+1)}{n} \right \rceil \leq \frac{4d(q+1)}{n}$ and note that by assumption $r^d < n/3$. Let $A \subset V$ be an arbitrary set of $r^d$ vertices, and let $B = V \setminus A$. Note that $|A| < n/2, |B| > n/2$. Waiter will play only on the edges between $A$ and $B$, and will guarantee that a vertex in $A$ will reach degree at least $d$.
Waiter will play $d$ ``mega rounds", and will maintain sets $A_i \subseteq A$ such that before mega round number $i$, the following holds:
\begin{itemize}
\item $|A_{i}| = r^{d - i + 1}$;
\item For all $v \in A_{i}$, the number of free edges from $v$ to $B$ is at least $\frac{n}{2d}(d-i+1)$;
\item For all $v \in A_{i}$, $deg_{G_C}(v) \geq i-1$.
\end{itemize}
Assuming the above, Waiter plays in the following way: during mega round $i$, Waiter arbitrarily divides $A_{i}$ into $r^{d - i}$ sets $S_1, \dots, S_{r^{d - i}}$ of $r$ vertices each. For each set $S_j$, Waiter offers $\left \lceil \frac{q+1}{r} \right \rceil$ edges from each vertex in $S_j$ to arbitrary vertices in $B$ (thus supposedly offering at least $q+1$ edges in total. Actually, Waiter can choose any arbitrary $q+1$ edges of these). Client then chooses one edge, which touches a vertex in $S_j$. Let $v_j$ be the vertex in $S_j$ whose edge was chosen by Client. Let $A_{i+1} = \{v_j\}_{j=1}^{r^{d - i}}$, and note that $|A_{i+1}| = r^{d - i}$. Also, every vertex $v_j \in A_{i+1}$ had at least $\frac{n}{2d}(d-i+1)$ free edges to $B$ before mega round $i$, and during mega round $i$ Waiter offered at most  $\left \lceil \frac{q+1}{r} \right \rceil < \frac{n}{3d} + 1 < \frac{n}{2d}$ edges containing $v_j$, so now it has at least $\frac{n}{2d}(d-i)$ free edges to $B$. Moreover, every vertex $v_j \in A_{i+1}$ satisfied $deg_{G_C}(v_j) \geq i-1$ before mega round $i$, and during mega round $i$ an edge containing it was chosen by Client, so $deg_{G_C}(v) \geq i$ before mega round $i+1$. This yields that indeed $A_{i+1}$ satisfies the three conditions above. Noting that $A_1 = A$ satisfies all the conditions for $i = 1$, we have our induction base and we can conclude that after $d$ mega rounds, Waiter is left with a set $A_{d+1}$ of size $r^0 = 1$, and the sole vertex $v \in A_{d+1}$ satisfies $deg_{G_C}(v) \geq d$, establishing the claim.
\end{proof}

Armed with Proposition \ref{large_bias_client} and Lemma \ref{large_bias_waiter}, we can now prove the main results for this section.

Proving Theorem \ref{large_bias_const} is quite straightforward: For Client's side, it is easy to see that when $q > \frac{e}{d} n^{1 + \frac{1}{d}}$, the assumption of Proposition \ref{large_bias_client} holds. For Waiter's side, when $q < \frac{1}{12d} n^{1 + \frac{1}{d}}$ the assumption of Lemma \ref{large_bias_waiter} holds.

Proving Theorem \ref{large_bias_inf} will also be simple, but a bit technical. What we claim is that, when the conditions of Theorem \ref{large_bias_inf} hold, then $D_n^q = (1 + o(1))d_n^q = (1+o(1))f_n^q$. Verifying this, and combining it with the Proposition \ref{large_bias_client} and Lemma \ref{large_bias_waiter} proves the theorem.

\begin{proposition}
Suppose $q = \omega\left(\frac{n}{\ln n}\right)$ and also $q = n^{1+o(1)}$. Then 
\[ D_n^q = (1 + o(1))d_n^q = (1+o(1))f(n,q) .\]
\end{proposition}

\begin{proof}
First of all, it is clear that $d_n^q \leq f(n,q) \leq D_n^q$ by their definitions. Therefore it suffices to show that for any fixed $\varepsilon > 0$ and large enough $n$ it holds that $D_n^q \leq (1+\varepsilon)d_n^q$.

We begin by proving $d_n^q = \omega(1)$. Suppose by contradiction that $d_n^q < C$ infinitely often for some constant $C > 0$. Then for infinitely many values of $n$ we have $3(4C(q+1))^C \geq n^{C+1}$. Since $q = n^{1+o(1)}$, we have $3(4C)^C \geq n^{C + 1 - C(1+o(1))}$, a contradiction for large enough $n$.

Next we claim that $d_n^q = o(\ln n)$. Suppose by contradiction that $d_n^q > c \ln n$ infinitely often for come constant $c > 0$. Then for infinitely many values of $n$ we have $3(4c \ln n(q+1))^{c \ln n} < n^{c \ln n+1}$. That is a contradiction because 
\[3(4c \ln n(q+1))^{c \ln n} > \left(4c \ln n \omega\left(\frac{n}{\ln n}\right)\right)^{c \ln n} = (4c\omega(n))^{c \ln n} > n^{c \ln n + 1}\]
for large enough $n$.

Now we fix $\varepsilon > 0$ and show that  $D_n^q \leq (1+\varepsilon)d_n^q$. Let $d = d_n^q + 1$ and recall that $(4d(q+1))^d > \frac{n^{d+1}}{3}$. Now we calculate (for large enough $n$):
\[ \left(\frac{(1+\varepsilon)d(q+1)}{e}\right)^{(1+\varepsilon)d} = \left(\frac{1+\varepsilon}{4e}\right)^{(1+\varepsilon)d} (4d(q+1))^{(1+\varepsilon)d} > \left(\frac{1+\varepsilon}{4e}\right)^{(1+\varepsilon)d} \left(\frac{n^{d+1}}{3}\right)^{1 + \varepsilon} \]
\[ = n^{(1+\varepsilon)d + 1} e^{a_\varepsilon + b_\varepsilon d + \varepsilon \ln n} > n^{(1 + \varepsilon) d + 1}\] where $a_\varepsilon, b_\varepsilon$ are constants depending on $\varepsilon$, and the last inequality due to $d = o(\ln n)$. From this it follows by definition that $D_n^q < (1+\varepsilon)(d_n^q + 1)$. Since $d_n^q = \omega(1)$, it also holds that $D_n^q < (1+2\varepsilon)d_n^q$ for large enough $n$, and we are done.
\end{proof}

\subsection{Small Bias}

In this subsection we will analyze the Waiter-Client Maximum Degree game played with a relatively small bias ($q = o\left(\frac{n}{\ln n}\right)$), proving tight results for the degree that Waiter can force, and that Client can avoid, playing $WC_\Delta(n, q, D)$. We then prove a simple result for $q = \Theta\left(\frac{n}{\ln n}\right)$, which is not as tight and has a multiplicative gap.

Our main goal is to prove the following result (as stated in Section \ref{introuction_section}):

\smallbias*

The only effort for this theorem is proving Client's side, as Waiter's side is trivial. In fact, no matter how Waiter plays, the number of edges in Client's graph is exactly $\left\lfloor \frac{n(n-1)}{2(q+1)} \right\rfloor$, yielding that the average degree in Client's graph is $(1+o(1))\frac{n}{q+1}$, and therefore some vertex has to have degree at least that large. So regardless of Waiter and Client's strategies, the maximum degree in Client's graph will be at least $(1+o(1))\frac{n}{q+1}$.

When we proved Theorem \ref{max_deg_upper}, we relied on ideas from \cite{Discrepancy Games} and adapted them to the Waiter-Client settings. These ideas seem to work perfectly for $q = 1$, but somehow generalizing them to work for other biases (even $q = 2$) proves difficult. In his book (\cite{Beck Monograph}), Beck has overcome this difficulty, and proved a result similar to \cite{Discrepancy Games} which holds for bias larger than $1$. We rely on his ideas and adapt them to the Waiter-Client setting, and then utilize them to provide Client with a strategy for the Waiter-Client Maximum Degree Game played with $q = o\left(\frac{n}{\ln n}\right)$. We will first present a general framework, and then use it specifically for the Maximum Degree game. We wish to note that the general framework we present can be further generalized, but it makes the proofs even more technical than they already are. Throughout the proofs, we will use similar notation to the notation used by Beck when proving his result (but not exactly the same, due to some differences). We begin by describing a general framework:

\begin{theorem}\label{biased_general_potential_framework}
Suppose Waiter and Client play a Waiter-Client game on a hypergraph $(X, \mathcal{F})$ with bias $q$, and let $\varepsilon, \alpha, \beta > 0$. Define 
\begin{align*}
E^+ &= \alpha\sum\limits_{k=0}^{\lfloor \frac{q}{2} \rfloor} \frac{\beta^{2k}}{2k+1}{q \choose 2k},
\\ E^- &= \beta - \sum\limits_{k=1}^{\lfloor \frac{q}{2} \rfloor}  \frac{\beta^{2k}}{2k}{q-1 \choose 2k-1} - \alpha \sum\limits_{k=1}^{\lfloor \frac{q}{2} \rfloor}  \frac{\beta^{2k}}{2k+1}{q-1 \choose 2k-1}. 
\end{align*}
If the following two conditions hold:
\[ E^+ \leq qE^- ;\]
\[ \sum_{A \in \mathcal{F}} \left((1+\alpha)^\frac{1+\varepsilon}{q+1}(1 - \beta)^\frac{q - \varepsilon}{q+1}\right)^{-|A|} < 1,\]

then Client has a strategy to avoid claiming at least $(1+\varepsilon)|A|$ elements of any edge $A \in \mathcal{F}$.
\end{theorem}

Note the similarity of the result to a Balancing game; Client indeed tries to avoid claiming \textbf{many} elements of an edge rather than his usual goal of avoiding claiming a full edge. We proceed to prove Theorem \ref{biased_general_potential_framework} and then apply it to deduce a result for the Waiter-Client Maximum degree game.

\begin{proof}
Client will use a potential method. After round $i$, we denote by $W_i, C_i$ the elements claimed by Waiter and Client respectively until (and including) round $i$. We define a potential function (after round $i$) for an edge $A \in \mathcal{F}$ to be $\psi_i(A) = (1+\alpha)^{|C_i \cap A| - \frac{1+\varepsilon}{q+1}|A|}(1-\beta)^{|W_i \cap A| - \frac{q-\varepsilon}{q+1}|A|}$. We define the potential of the game after round $i$ to be $\Psi_i = \sum\limits_{A \in \mathcal{F}}\psi_i(A)$. We also define the potential (after round $i$) of a set of elements $S \subseteq X$ to be $\psi_i(S) = \sum\limits_{A \in \mathcal{F}: S \subseteq A} \psi_i(A)$. Note that the second condition in the theorem statement exactly asserts that $\Psi_0 < 1$. For singletons, we will use $\psi_i(z)$ instead of $\psi_i(\{z\})$. 

As in most potential based strategies, we will show that Client can guarantee the potential does not increase during the game. Note that in the end of the game (say, after $r$ rounds), if Client claimed $(1+\varepsilon)|A|$ elements of some edge $A$, then $\psi_r(A) \geq 1$ and thus $\Psi_r \geq 1$. Therefore, by keeping the potential below $1$ during the entire game (and in the end of the game), Client essentially ensures he will not have so many elements in any edge. It is easy to see that during the last round, where Waiter claims all remaining unclaimed elements, the potential does not increase.
We now analyze the behavior of the potential $\Psi$ during the game. Suppose at round $i$ Waiter offered elements $z_1, \dots, z_{q+1},$ and without loss of generality Client claimed $z_1$. Denote $T = \{z_2, \dots, z_{q+1}\},$ and note that by a simple inclusion-exclusion type of calculation, we have:
\vspace{-.2cm}
\begin{align*}
\Psi_{i+1} &= \sum_{A \subset \mathcal{F}} \Psi_{i+1}(A) = \sum_{S \subseteq T} \left( \sum_{\substack{A : A \cap T = S \\ z_1 \in A}} (1+\alpha)(1-\beta)^{|S|} \Psi_i (A) +  \sum_{\substack{A : A \cap T = S \\ z_1 \notin A}} (1-\beta)^{|S|} \Psi_i (A)\right) 
\\ &= \sum_{S \subseteq T} \left( \sum_{\substack{A : S \subseteq A \\ z_1 \in A}} (1+\alpha)(-\beta)^{|S|} \Psi_i (A) +  \sum_{\substack{A : S \subseteq A \\ z_1 \notin A}} (-\beta)^{|S|} \Psi_i (A)\right) 
\\ &= \sum_{S \subseteq T} \left( \Psi_i(S)(-\beta)^{|S|} + \Psi_i(\{z_1\} \cup S)\alpha(-\beta)^{|S|} \right)
\\ &= \Psi_i + \sum_{\emptyset \neq S \subseteq T} \psi_i(S)(-\beta)^{|S|} +\sum_{S \subseteq T} \psi_i(\{z_1\} \cup S)\alpha(-\beta)^{|S|},
\end{align*}
where the third equality is due to the fact that $(1-\beta)^{|S|} = \sum\limits_{R \subseteq S} (-\beta)^{|R|}$ for any $S$.
 
Noting that, by definition, $\psi_i(S) \leq \frac{1}{|S|}\sum\limits_{z \in S} \psi_i(z)$, we can estimate:
\vspace{-.2cm}
\begin{align*}
\Psi_{i+1} &\leq \Psi_i + \alpha\psi_i(z_1)\sum_{S \subseteq T} \frac{(-\beta)^{|S|}}{|S| + 1} + \sum_{j=2}^{q+1}\psi_i(z_j) \sum_{S \subseteq T : z_j \in S} \left(\frac{(-\beta)^{|S|}}{|S|} + \alpha\frac{(-\beta)^{|S|}}{|S|+1}\right)
\\ &\leq \Psi_i + \alpha\psi_i(z_1)\sum_{k=0}^{\lfloor \frac{q}{2} \rfloor} \frac{\beta^{2k}}{2k + 1}{q \choose 2k} +  \sum_{j=2}^{q+1}\psi_i(z_j)\left(-\beta + \sum\limits_{k=1}^{\lfloor \frac{q}{2} \rfloor}  \frac{\beta^{2k}}{2k}{q-1 \choose 2k-1} + \alpha \sum\limits_{k=1}^{\lfloor \frac{q}{2} \rfloor}  \frac{\beta^{2k}}{2k+1}{q-1 \choose 2k-1}\right)
\\ &= \Psi_i + \psi_i(z_1)E^+ - \sum_{j=2}^{q+1}\psi_i(z_j)E^-.
\end{align*}
(Note that we get to the second inequality simply by omitting all of the negative terms except for the ones with $-\beta$ coefficient (the terms with $S$ such that $|S|$ is odd).)

Now it is easy to describe Client's strategy - when Waiter offers elements $z_1, \dots, z_{q+1}$, Client claims the element $z_j$ such that $\psi_i(z_j)$ is the smallest. Indeed, if w.l.o.g. we let  $z_1$ be the element which minimizes $\psi_i(z_j)$, plugging it in the previous equation together with the theorem's assumption that $E^+ \leq qE^-$ we have

\[ \Psi_{i+1} - \Psi_i \leq \psi_i(z_1)E^+ - \sum_{j=2}^{q+1}\psi_i(z_j)E^- \leq \psi_i(z_1)(E^+ - qE^-) \leq 0 \]
and therefore Client can guarantee that the potential does not increase. When justifying the above inequality, note the technical detail that $E^+ \geq 0$ by definition, and since $E^+ \leq qE^-$ we also have $E^- \geq 0$. As explained in the beginning of the proof, this is enough for Client to avoid claiming $\frac{(1+\varepsilon)}{q+1}|A|$ elements of any edge $A \in \mathcal{F}$, and thus we are done.
\end{proof}

We now have the general framework we need. All it takes from here is to set the correct $\varepsilon, \alpha, \beta$ and to get the desired result. For the sake of simplicity, we do not optimize constants; even without it the result gives an asymptotic estimate. We now prove Theorem \ref{small_bias_maximum_degree}:

\begin{proof}[Proof of Theorem \ref{small_bias_maximum_degree}]
Let us set $\varepsilon = 4\sqrt{\frac{(q+1)\ln n}{n}}, \alpha = (1+o(1))\frac{q+1}{5q+1}\varepsilon, \beta = \frac{\alpha + 2\alpha^2}{q}$, where the $1+o(1)$ in the choice of $\alpha$ will be determined later. Note that by our assumption on $q$, we have $\varepsilon = o(1)$ and therefore $\alpha = o(1), \beta q = o(1)$ as well.

We claim both conditions of Theorem \ref{biased_general_potential_framework} hold. We start by showing that $E^+ \leq qE^-$:
\vspace{-.2cm}
\begin{align*}
&E^+ = \alpha\sum\limits_{k=0}^{\lfloor \frac{q}{2} \rfloor} \frac{\beta^{2k}}{2k+1}{q \choose 2k} \leq \alpha\sum\limits_{k=0}^{\lfloor \frac{q}{2} \rfloor} (\beta q)^{2k} \leq \frac{\alpha}{1 - (\beta q)^2}.
\\ &E^- = \beta - \sum\limits_{k=1}^{\lfloor \frac{q}{2} \rfloor}  \frac{\beta^{2k}}{2k}{q-1 \choose 2k-1} - \alpha \sum\limits_{k=1}^{\lfloor \frac{q}{2} \rfloor}  \frac{\beta^{2k}}{2k+1}{q-1 \choose 2k-1} \geq \beta \left(1 - (1+\alpha)\sum\limits_{k=1}^{\lfloor \frac{q}{2} \rfloor}  (\beta q)^{2k-1}\right) \geq \beta\left(1-\frac{(1+\alpha)\beta q}{1 - (\beta q)^2}\right).
\\ & \Rightarrow \left(1 - (\beta q)^2\right) (qE^- - E^+) \geq   \beta q \left(1 - (\beta q)^2 - (1+\alpha)\beta q\right) - \alpha = \alpha^2 + O(\alpha ^ 3) > 0.
\end{align*}

Next we show that the second condition of the theorem holds. First we calculate:
\vspace{-.2cm}
\begin{align*}
& \ln \left((1+\alpha)^{1+\varepsilon}(1 - \beta)^{q - \varepsilon}\right) = (1+\varepsilon)\ln (1 + \alpha) + (q - \varepsilon) \ln (1 - \beta) 
\\ & = (1 + \varepsilon)\left(\alpha - \frac{\alpha^2}{2} + O(\alpha^3)\right) - (q - \varepsilon)\left(\beta + \frac{\beta^2}{2} + O(\beta^3)\right) = -U\alpha^2 + V\alpha + O(\alpha^3)
\\ & \text{where } U = \frac{5q + 1}{2q}(1+o(1)), V = \frac{q+1}{q}\varepsilon.
\end{align*}

By choosing $\alpha = \frac{V}{2U} = (1+o(1))\frac{q+1}{5q+1}\varepsilon$ and plugging it in, we get that 
\[ \ln \left((1+\alpha)^{1+\varepsilon}(1 - \beta)^{q - \varepsilon}\right) = \frac{V^2}{4U} + O(\alpha^3) \geq (1+o(1))\frac{\varepsilon^2}{12}, \]
with the $(1+o(1))$ due to the fact that $O(\alpha ^3) = O(\varepsilon^3)$.

Note that in the maximum degree game, every edge $A \in \mathcal{F}$ is of size $|A| = n - 1$, and also $|\mathcal{F}| = n$. Therefore, the second condition of Theorem \ref{biased_general_potential_framework} holds if
\[ \left((1+\alpha)^{1+\varepsilon}(1 - \beta)^{q - \varepsilon}\right)^{-\frac{n-1}{q+1}} < \frac{1}{n}.\]
Taking the logarithm of both sides, we need to show that 
\[ \ln \left((1+\alpha)^{1+\varepsilon}(1 - \beta)^{q - \varepsilon}\right) > \frac{(q+1)\ln n}{n-1}. \]
Indeed, we have
\[ \ln \left((1+\alpha)^{1+\varepsilon}(1 - \beta)^{q - \varepsilon}\right) \geq (1+o(1))\frac{\varepsilon^2}{12} = (1+o(1))\frac{16}{12} \frac{(q+1)\ln n}{n}> \frac{(q+1)\ln n}{n-1}, \]
completing the proof.

\end{proof}

Note that our proof of Theorem \ref{small_bias_maximum_degree} actually gives a specific error term, as $\varepsilon$ is a specific function of $q$ and $n$. The smaller the value of $q$ is, the tighter our bound becomes for the $o(1)$ part. In particular, for constant $q$ we get a $\frac{n}{q+1} + O_q(\sqrt{n \ln n})$ upper bound. For $q=1$, this gives the same asymptotic result as in Theorem \ref{max_deg_upper}. However, we proved Theorem \ref{max_deg_upper} differently, mainly because it is simpler and gives a specific and relatively decent multiplicative constant of $\frac{1}{\sqrt{2}}$ in front of the error term.

When looking at the results for large and small bias, random graphs immediately come to mind. The maximum degree of the binomial random graph $G(n,p)$ and the maximum degree of the random graph $G(n,m)$ with $m = \left \lfloor p {n \choose 2} \right \rfloor$ both typically behave very similarly and are well studied (see \cite{Introduction to Random Graphs}, \cite{Random Graphs}). Their behavior also nicely matches the behavior of the maximum degree in Client's graph in a perfectly played biased Waiter-Client Maximum Degree game. Indeed, when $p = o\left(\frac{\ln n}{n}\right)$, which corresponds to $q = \omega\left(\frac{n}{\ln n}\right)$, the random graph behavior matches that of Theorems \ref{large_bias_inf}, \ref{large_bias_const}. When $p = n^{-1 + o(1)}$ the maximum degree of $G(n,p)$ is of the same magnitude as the one attained by Theorem \ref{large_bias_inf}, and when $p = \Theta\left(n^{- 1 - \frac{1}{d}}\right)$, the maximum degree of $G(n,p)$ is w.h.p. either $d$ or $d - 1$, perfectly matching the transition between maximum degree $d-1$ and $d$ appearing in the Waiter-Client game when $q = \Theta\left(n^{1 + \frac{1}{d}}\right)$ given by Theorem \ref{large_bias_const}. When $p = \omega\left(\frac{\ln n}{n}\right)$, which corresponds to $q = o\left(\frac{n}{\ln n}\right)$, the random graph $G(n,p)$ has w.h.p. a maximum degree of $(1+o(1))np$, exactly matching Theorem \ref{small_bias_maximum_degree}. Thus we see that the Waiter-Client Maximum Degree game, when played perfectly, behaves very much like it does when played randomly.

The natural question to address is what happens when $q = \Theta\left(\frac{n}{\ln n}\right)$? For the random graph, it is some kind of a transition phase between having asymptotically equal maximum and average degrees and having asymptotically different maximum and average degrees. We see the same behavior in the Waiter-Client Maximum Degree game, and give the following result:

\begin{proposition}
Suppose $q = c\frac{n}{\ln n}$ for some constant $c > 0$. Then Waiter can force Client to have maximum degree at least $(1+o(1))\frac{n}{q+1}$, while Client can ensure his maximum degree is at most $C\frac{n}{q+1}$ for some $C > 0$ (which depends only on $c$).
\end{proposition}

\begin{proof}
For Waiter's side, we already established after stating Theorem \ref{small_bias_maximum_degree} that it holds regardless of Client and Waiter's strategies. For Client's side, we utilize Proposition \ref{large_bias_client} and see that setting $d = C\frac{n}{q+1}$ satisfies the condition whenever $\frac{C(\ln C - 1)}{c} \geq 1$.
\end{proof}

We did not make an effort in trying to optimize (or even explicitly state some of) the constants in the previous proposition. We only tried to give a taste of the fact that during the transition phase $q = \Theta \left(\frac{n}{\ln n}\right)$ the maximum degree in the game is still somewhat controlled.

\section{Concluding remarks and open questions}

In this paper we have studied the Waiter-Client Maximum Degree game. We first focused on the unbiased case $q = 1$, where it is easy to show that Client can avoid maximum degree of $(1+o(1))\frac{n}{2}$ while Waiter can force Client to have a maximum degree of the same asymptotic value. Our interest was in the error term, partly because it is still an open question for the Maker-Breaker Maximum Degree game. We managed to determine the correct magnitude of the error term, showing that Client can avoid a maximum degree of $\frac{n}{2} + (1+o(1))\sqrt{\frac{n \ln n}{2}}$, while Waiter can force a maximum degree of $\frac{n}{2} + c\sqrt{n \ln n}$ for some $c > 0$. Naturally, we are interested in the correct constant in front of the $\sqrt{n \ln n}$ expression.

\begin{question}
What is the constant $c$ such that playing the unbiased Waiter-Client Maximum Degree game, Waiter can force Client to have maximum degree of $\frac{n}{2} + (1+o(1))c\sqrt{n \ln n}$ while Client can ensure his maximum degree is at most $\frac{n}{2} + (1+o(1))c\sqrt{n \ln n}$?
\end{question}

Another natural question is what happens when $q$ is constant, but not $1$. The most simple example is, of course, $q=2$. It is easy to see that the maximum degree (assuming perfect play) in this case will be $(1+o(1))\frac{n}{q+1}$ (in fact it can be deduced from Theorem \ref{small_bias_maximum_degree}). The question is, what is the magnitude of the error term? 

\begin{question}
Let $q \geq 2$ be a fixed integer. What is the correct error term $\gamma(n)$ such that playing the Waiter-Client Maximum Degree game, Waiter can force Client to have a maximum degree of at least $\frac{n}{q+1} + c\gamma(n)$ (for some $c>0$) while Client can ensure his maximum degree is at most $\frac{n}{q+1} + C\gamma(n)$ (for some $C>0$)?
\end{question}

We believe that the error term should also be $\Theta_q(\sqrt{n \ln n})$ just like in the unbiased case, but we were not able to prove it. The upper bound (Client's side) can be deduced by Theorem \ref{small_bias_maximum_degree}, but for the lower bound our approach needs at least some tweaking to work. One can define a process similar to the Advantage Process defined in this paper, and analyze it when trying to answer this question. We believe understanding the Advantage Process better will ultimately solve this question entirely.

Next we studied the biased case and showed tight asymptotic results for the cases $q = o\left(\frac{n}{\ln n}\right)$ and $q = \omega\left(\frac{n}{\ln n}\right)$, except for certain intervals $q \in \left[c_dn^{1 + \frac{1}{d}}, C_dn^{1 + \frac{1}{d}}\right]$. In the case $q = o\left(\frac{n}{\ln n}\right)$ our result actually yields specific bounds for the error terms (the difference between the outcome of perfect play and $\frac{n}{q+1}$). Naturally, it is interesting to understand more about what happens in the ``transition phase" $q = \Theta\left(\frac{n}{\ln n}\right)$.

\begin{question}
Suppose $q = t\frac{n}{\ln n}$ for some constant $t > 0$. What is the constant $c = c(t)$ such that playing the Waiter-Client Maximum Degree game with bias $q$, Waiter can force Client to have maximum degree at least $(1+o(1))c\frac{n}{q+1}$ while Client can ensure that his maximum degree is at most $(1+o(1))c\frac{n}{q+1}$?
\end{question}

All of our results show that the maximum degree in Client's graph, when Waiter and Client play perfectly, behaves very similarly to how it behaves if they play randomly. We have shown that the \textit{probabilistic intuition} correctly predicts the results of the Waiter-Client Maximum Degree game. Answering the previous question will give us the complete picture on the matter, and it is very interesting to know whether in the transition phase the maximum degree attained by perfect play behaves similarly to the one attained by random play.

\section*{Acknowledgements}

We would like to thank Ohad Klein for sharing his useful thoughts on the Advantage Process.

\end{document}